\documentclass[preprint,12pt,number, sort&compress]{elsarticle}
\usepackage{pgf}
\usepackage{multicol}
\usepackage{numcompress}
\usepackage[left=0.72in,right=0.72in, bottom=0.98in, top=0.97in]{geometry}
\usepackage{amssymb,amsmath}
\usepackage{tikz}
\usetikzlibrary{arrows}
\usetikzlibrary{decorations.pathreplacing,decorations.markings}
\tikzset{arrow data/.style 2 args={%
		decoration={
		markings,
			mark=at position #1 with \arrow{#2}},
		postaction=decorate}
}%
\usepackage{caption}
\usepackage{float}
\usepackage[labelfont=bf,justification=raggedright,singlelinecheck=false]{caption}
\usepackage{caption}
\usepackage{subcaption}
\usepackage{float}
\usepackage{url}

\newcommand*{\doi}[1]{\href{http://dx.doi.org/\detokenize{#1}}{doi: \detokenize{#1}}}
\usepackage{lipsum}
\usepackage{dblfloatfix}
\usepackage{array}
\usepackage{booktabs}
\usepackage{threeparttable}
\usepackage[referable]{threeparttablex}
\usepackage{graphics,graphicx}
\usepackage{enumerate}
\usepackage{amsfonts}
\usepackage{adjustbox}
\usetikzlibrary{decorations.markings}
\usepackage[T1]{fontenc}
\usepackage[colorlinks=true,breaklinks]{hyperref}
\hypersetup{
	colorlinks=true,
	citecolor=blue
	linkcolor=blue
	filecolor=blue    
	urlcolor=blue,
}
\usepackage{color}
\newcommand{\R}{\mathbb{R}}
\newcommand{\N}{\mathbb{N}}
\usepackage[capitalize,nameinlink]{cleveref}
\DeclareUrlCommand\doi{\def\UrlLeft##1\UrlRight{doi:\href{http://dx.doi.org/##1}{##1}}\urlstyle{rm}}
\usepackage{siunitx}

\usepackage{amsmath}
\usepackage[utf8]{inputenc}
\numberwithin{equation}{section}
\usepackage{amsthm}
\theoremstyle{plain}
\newtheorem{cor}{Corollary}
\theoremstyle{plain}
\newtheorem{thm}{Theorem}[section]
\newtheorem{pr}{Proposition}[section]
\theoremstyle{definition}

\newtheorem{remark}{Remark}[section]

\newtheorem{lm}{Lemma}[section]
\allowdisplaybreaks
\usepackage{nameref}
\usepackage[english]{babel}
\usepackage{xpatch}
\xpatchcmd{\proof}{\itshape}{\normalfont\proofnameformat}{}{}
\newcommand{\proofnameformat}{}
\renewcommand{\proofnameformat}{\bfseries}
\usepackage{booktabs,xcolor,siunitx}
\definecolor{lightgray}{gray}{0.9}
\usepackage{lmodern}
\usepackage{breqn}

\journal{Journal}

\begin{document}	
	\begin{frontmatter}
		
		
		
		\title{Bifurcations of limit cycles in piecewise smooth Hamiltonian system with boundary perturbation}
		
		 \author[label1]{Nanasaheb Phatangare}
		 \address[label1]{Department of Mathematics, Fergusson College, Pune}
		 \author[label2]{Krishnat Masalkar}
		 \address[label2]{Department of Mathematics, Abasaheb Garware College, Pune}
		 \author[label3]{Subhash Kendre\corref{mycorrespondingauthor}}
		 \address[label3]{Department of Mathematics, Savitribai Phule Pune University, Pune}\cortext[mycorrespondingauthor]{Corresponding author}
		 
		
		
		
		\begin{abstract}
			In this paper, the general planar piecewise smooth Hamiltonian system with period annulus around the center at the origin  is considered. We obtain the expressions for the first order and the second order Melnikov functions of it's general second order perturbation, which can be used to find the number of limit cycles bifurcated from  periodic orbits.
			Further, we have shown that the number of limit cycles of the system   $\dot{X}=\begin{cases}
		(H_y^+,-H_x^+) & \mbox{if}~y>\varepsilon f(x)\\ (H_y^-,-H_x^-) & \mbox{if}~y<\varepsilon f(x)
		 \end{cases}$ equals to the number of positive zeros of $f$ when at $\varepsilon=0$ the system has a period annulus around the origin. 
		   \end{abstract}
		
		
	\end{frontmatter}

	
	\section{Introduction}
	
	Motions of many nonsmooth processes such as impact switching, sliding and other discrete state transitions are modelled into piecewise smooth dynamical systems rather than the smooth dynamical systems. Recently piecewise smooth dynamical systems are of great interest. It has many applications in physical processes  such as electrical circuits, impact oscillators, dry friction oscillators, relay control systems, modelling of irregular heartbeats etc.\cite{bernardo2008piecewise}. In many scientific applications systems with self sustained oscillations are modelled where limit cycles plays an important role. Limit cycles bifurcations in case of smooth dynamical systems is very well studied, whereas the non-smooth systems have been studied recently. 
	\par Averaging theory, Melnikov theory and normal form theory are well known techniques used  to study the limit cycle bifurcation of planar smooth differential systems \cite{han2012normal, sanders2007averaging}, whereas the techniques for piecewise smooth systems are in the process of development \cite{du2008bifurcation,llibre2015averaging}. 
	\par  In \cite{de2013limit} authors considered a piecewise linear differential systems (PLDS) having center-focus type singularity with switching manifold $y=0$, in which limit cycle bifurcation of the system is studied when the switching manifold is $y=\varepsilon$. Also, in \cite{zou2018piecewise} C. Zou and J. Yang studied PLDS with saddle-centre type singularity at the origin and switching curve $y=b\sin x$, in which it has been shown that the number of limit cycles  bifurcated from the period annulus of the system with $b=0$ is equal to the number of positive zeros of $\sin x$. Note that the system  considered in \cite{zou2018piecewise} is symmetric about the $y$-axis and zeros of switching curve $y=b \sin x$ are also  symmetric about the $y$-axis. In \cite{zou2019piecewise} authors studied the same system as in \cite{zou2018piecewise} by considering the switching curve $y=bx(x^2-x_1^2)(x^2-x_2^2)\cdots (x^2-x_m^2)$, wherein it is proved that the number of limit cycles bifurcated from the period annulus of the system at $b=0$ is equal to $m$, where $m$ is a positive integer.	
	 In \cite{buzzi2014birth}, authors studied the number of limit cycles bifurcated from the origin of the perturbation of a planar piecewise smooth system  with centre-centre type singularity  at the origin. Further, in \cite{wei2016normal}, normal forms of some planar piecewise smooth systems  with center-center type singularity of order $(k,l)$ at the origin are considered and their limit cycles bifurcation from the origin under higher order perturbations have been studied. It is natural to think about the limit cycles bifurcation of these normal forms when the separation boundary is an analytic function. 
	In \cite{liu2010bifurcation} Xia Liu and Maoan Han considered the first order perturbation of a planar piecewise smooth Hamiltonian system. If the unperturbed system has a period annulus centered at the origin, then using the first order Melnikov function, the number of limit cycles bifurcated from the periodic annulus are studied.
	\par In this paper we have obtained the second order Melnikov function for the piecewise Hamiltonian system with second order perturbation and the separation boundary $y=0$. We also considered a general piecewise smooth Hamiltonian system with perturbed separation boundary $y=\varepsilon f(x)$ when $f$ is a $C^2$ function. 
  	 \par The paper is organized as follows: In \textit{Section 2} we give some perliminary concepts about Melnikov theory, limit cycles and stability of limit cycles. \textit{Section 3} is devoted to investigate the first order and second order Melnikov functions for piecewise smooth Hamiltonian systems with second order perturbation. \textit{Section 4} deals with the general piecewise smooth Hamiltonian system with boundaty perturbation. Finally, in \textit{Section 5}, we give some application of piecewise smooth Hamiltonian systems with boundary perturbation. 
\section{Preliminaries}
 Consider a $ C^{\infty}$ smooth system of the form \begin{align}\label{p121}\dot{X}=\begin{cases}
  H_y+\varepsilon f(x,y,\varepsilon, \delta)\\
  -H_x+\varepsilon g(x,y,\varepsilon, \delta)
  \end{cases},
  \end{align} where $H,f,g$ are $C^{\infty}$ smooth functions for $\varepsilon \in \R, \delta \in K\subset \R^m$ with $K$ compact and $H_x(x,y)=\dfrac{\partial H}{\partial x}(x,y), H_y(x,y)=\dfrac{\partial H}{\partial y}(x,y)$. For $\varepsilon =0$, the system (\ref{p121}) becomes Hamiltonian system \begin{align}\label{p122}
  \dot{X}=\begin{cases}
  H_y\\-H_x
  \end{cases}.
  \end{align} 
  Suppose that the system (\ref{p122}) has a period annulus $\displaystyle \mathcal{A }=\{\Gamma_h:H(x,y)=h,h\in (\alpha, \beta)\subset \R\}$ such that  $ \Gamma_h \rightarrow \Gamma_{\alpha}$ as $h\rightarrow \alpha,$ which is an elementary center point for the system and $ \Gamma_h \rightarrow \Gamma_{\beta}$ as $h\rightarrow \beta,$ which is usually a homoclinic loop consisting of a saddle point or heteroclinic loop containing two saddle points. For some $h_0\in (\alpha, \beta)$, consider a periodic orbit $\Gamma_{h_0}$ from the period annulus and a Poincare section $S=\{(a(h),0):h\in (h_0-\varepsilon_0,h_0+\varepsilon_0)\}$, for some $\varepsilon_0>0$, at the point $A(a(h_0),0)$ to $\Gamma_{h_0}$. Let $\Gamma_{h_0\varepsilon} $ be the solution of (\ref{p121}) starting at $A(a(h_0),0)$and let $B(b(h_0,\varepsilon, \delta),0)$ be its first point of intersection with the Poincare section. Then the Poincare map $\mathcal{P}$ maps $A(a(h_0),0)$  to $B(b(h_0,\varepsilon, \delta),0)$. Note that $H(A(a(h_0),0))=H(B(b(h_0,\varepsilon, \delta),0))$ if and only if $A(a(h_0),0)=B(b(h_0,\varepsilon, \delta),0)$. Therefore we can use difference map $H(B)-H(A)$ to investigate the number of limit cycles of (\ref{p121}) bifurcated from $\Gamma_{h_0}$.  Thus, 
\begin{align}\label{p123}
H(B)-H(A)
=&\int_{\widehat{AB}}dH
=\int_{\widehat{AB}}(H_xdx+H_ydy)\nonumber \\
=&\int_0^{\tau_{h_0}}[H_x(H_y+\varepsilon f)+H_y(-H_x+\varepsilon g)]dt  
=\varepsilon \int_0^{\tau_{h_0}}(H_x f+H_y g)dt\nonumber\\
=& \varepsilon F(h_0,\varepsilon ,\delta) 
= \sum_{k=1}^{\infty}M_k(h_0,\delta)\varepsilon ^k,
\end{align}
 where $\tau_{h_0}$ is the  time of flight along the trajectory $\widehat{AB}$ of (\ref{p121}) from $A$ to $B$ and 
\begin{align}
M_k(h_0,\delta)=\dfrac{1}{(k-1)!}\dfrac{\partial^{(k-1)} F}{\partial \varepsilon ^{k-1}}(h_0,0,\delta). \nonumber
\end{align}
Here, $M_k(h_0,\delta)$ is called as the $k$th order Melnikov function and $F$ is called as a bifurcation function for the system (\ref{p121}).
\par Clearly, from equation (\ref{p123}) we have 
\begin{align}
F(h_0,0,\delta)=M_1(h_0,\delta)=&\int_{\Gamma_{h_0}}(H_xf+H_yg)dt=\int_{\Gamma_{h_0}}(gdx-fdy)=-\int\int_{Int(\Gamma_{h_0})}(f_x+g_y)dxdy, \nonumber
\end{align} where $Int(\Gamma_{h_0})$ is the region bounded by $\Gamma_{h_0}$.
Here, we say that the cyclicity of $\Gamma_{h_0}$ is $k$ if  there exist $\varepsilon_0$ such that  (\ref{p121}) has at most $k$ limit cycles in some neighborhood of $\Gamma_{h_0}$ for any $\delta\in K$	and for any $0<\varepsilon<\varepsilon_0$ and  that (\ref{p121}) has exactly $k$ limit cycles in every neighbourhood of $\Gamma_{h_0}$ for some  $(\varepsilon, \delta)$.
\par The following proposition states that the number of periodic solutions of (\ref{p121}), called as limit cycles, in small neighbourhood for $\Gamma_{h_0}$ is less than or equal to the number of isolated zeros of the first order Melnikov function $M_1(h_0,\delta)$.
 \begin{pr}\label{p12P1}\cite{perko2013differential} 
 Let $ \delta_0\in K$. Then we have the following:
 \begin{enumerate}
 \item There is no limit cycle near $\Gamma_{h_0}$ for $\varepsilon+|\delta-\delta_0|$ small, if $M_1(h_0,\delta_0)\neq 0$.
 \item There is exactly one (at least one ) limit cycle $\Gamma(h_0,\varepsilon, \delta)$ for $\varepsilon+|\delta-\delta_0|$, which approaches $\Gamma_{h_0}$ as $(\varepsilon, \delta)\rightarrow (0, \delta_0)$ if $M_1(h_0, \delta_0)=0, ~\dfrac{\partial M_1}{\partial h}(h_0, \delta_0)\neq 0$ (if $h_0$ is a zero of $M(h, \delta_0)$ with odd multiplicity, respectively).
 \item If there exist $0\leq j\leq k$ such that $M_1(h_0,\delta_0)=0$ and $\dfrac{\partial^j M_1}{\partial h^j}(h_0,\delta)\neq 0$ then atmost $k$ limit cycles of (\ref{p121}) are bifurcated form $\Gamma_{h_0}$. 
 \end{enumerate}
  \end{pr}
Now we have the following result about the stability  of limit cycles using the first order Melnikov function.
\begin{pr}\label{p12P2}
 The limit cycle of (\ref{p121}) bifurcated from the periodic orbit of (\ref{p122}) passing through $(a(h),0)$ of the Poincare section  is stable if and only if $$\dfrac{dM_1}{dh}-M_1\dfrac{H_{xx}(a(h),0)}{H_x(a(h),0)^2}<0,$$
where $M_1$ is the first order Melnikov function for (\ref{p121}).
\end{pr}
\begin{proof} From equation (\ref{p123}), we have
\begin{align}
H(b(h,\varepsilon, \delta),0)-H(a(h),0)=\varepsilon M_1(h,\delta)+o(\varepsilon^2).\nonumber
\end{align}
By Taylor's expansion in powers of $\varepsilon$, we have
\begin{align}\label{p124} 
H(a(h),0)+\varepsilon H_x(a(h),0)\left(\dfrac{\partial b}{\partial \varepsilon}\right)_{\varepsilon=0}+ o(\varepsilon^2)-H(a(h),0)
&=\varepsilon M_1(h, \delta)+o(\varepsilon^2).
\end{align}
Equating $\varepsilon$ order terms in equation (\ref{p124}) on both sides we get
\begin{align}
H_x(a(h),0)\left(\dfrac{\partial b}{\partial \varepsilon}\right)_{\varepsilon=0}
&=M_1(h,\delta).\nonumber
\end{align}
 Now if $\mathcal{P}_{\varepsilon}$ is the Poincare map of system (\ref{p121}) then we have $\mathcal{P}_{\varepsilon}(a(h))=b(h,\varepsilon,\delta)$. Hence 
 \begin{align}\label{p125}\left(\dfrac{\partial \mathcal{P}_{\varepsilon}}{\partial \varepsilon}\right)_{\varepsilon=0}=\dfrac{M_1(h,\delta)}{H_x(a(h),0)}.
 \end{align}
Now differentiang (\ref{p125}) with respect to $h$ we get, 
\begin{align}\label{p126}
\dfrac{d}{dx}\left[\left(\dfrac{\partial \mathcal{P}_{\varepsilon}}{\partial \varepsilon}\right)_{\varepsilon=0}\right]a'(h)&=\dfrac{1}{H_x(a(h),0)}\dfrac{dM_1}{dh}-M_1\dfrac{H_{xx}(a(h),0)}{(H_x(a(h),0))^2}a'(h).
\end{align}
Since $H(a(h),0)=h$, we have $H_x(a(h),0)a'(h)=1$. Therefore from (\ref{p126}) we get,
\begin{align*}
\dfrac{\partial}{\partial \varepsilon}\left[\left(\dfrac{d\mathcal{P}_{\varepsilon}}{dx}\right)\right]_{\varepsilon=0}&=\dfrac{dM_1}{dh}-M_1\dfrac{H_{xx}(a(h),0)}{H_x(a(h),0)^2}.
\end{align*}
Now if $0<\varepsilon \ll 1$, then we have 
\begin{align*}
\dfrac{d\mathcal{P}_{\varepsilon} }{dx}-\dfrac{d\mathcal{P}_{0}}{dx} &\approx\varepsilon\left (\dfrac{dM_1}{dh}-M_1\dfrac{H_{xx}(a(h),0)}{H_x(a(h),0)^2}\right).
\end{align*} 
But $\mathcal{P}_0$ is poincare return map for (\ref{p122}). Hence, $\mathcal{P}_0(a(h))=a(h)$, which imply that  $$\dfrac{d\mathcal{P}_{0}}{dx}(a(h))a'(h)=a'(h).$$ Hence,  
\begin{align*}
\dfrac{d\mathcal{P}_{\varepsilon} }{dx}-1 &\approx\varepsilon\left (\dfrac{dM_1}{dh}-M_1\dfrac{H_{xx}(a(h),0)}{H_x(a(h),0)^2}\right).
\end{align*} 
Thus, the limit cycle passing through $A(a(h),0)$ is stable if and only if $$ \dfrac{dM_1}{dh}-M_1\dfrac{H_{xx}(a(h),0)}{H_x(a(h),0)^2}<0.$$
\end{proof} 

\par General planar piecewise smooth differential system with two zones and switching manifold $\Sigma=\varphi^{-1}(0)$ is given by
\begin{align}\label{p127}\dot{X}=&\begin{cases}
( X_1^+(x,y),~~X_2^+(x,y)), ~~~(x,y)\in \Sigma^{+}\\
( X_1^-(x,y),~~X_2^-(x,y))~~~(x,y)\in \Sigma^{-}
\end{cases},
\end{align}
where $X_i^{\pm},f^{\pm},g^{\pm},i=1,2$  and $\varphi$  are sufficiently smooth functions on some open region containing the origin with $0$ as a regular value of $\varphi$ and $\Sigma^+=\varphi^{-1}(0,\infty),~~\Sigma^-=\varphi^{-1}(-\infty,0)$ are two zones of (\ref{p127}).
\par Now we denote $X^{\pm}=:(X_1^{\pm}, X_2^{\pm})$ ,
 $X^{\pm}\varphi=:<X^{\pm}, \nabla \varphi>$ and  $(X^{\pm})^k\varphi=:<X^{\pm},\nabla (X^{\pm})^{k-1}\varphi>$, where $<,>$ is an Euclidean dot product.
\par We say that a point $p\in \Sigma$ is a $k$th contact point for the vector field $X$ if  $(X^k\varphi)(p)\neq 0$ and $(X^{l}\varphi)(p) =0$ for $l=1,\cdots, k-1$.
 A point $p\in \Sigma$ is a $(k, l)$-contact singularity of $X^{\pm}$ if $X_0$ is a $k$th contact point for $X^{+}$ and is a  $l$th contact point for $X^{-}$.
\section{Perturbation of Piecewise Smooth Hamiltonian System}
\par 
Recently in \cite{liu2010bifurcation}, authors studied the number of limit cycles of the peicewise smooth Hamiltonian system
\begin{align}\label{p131}
	\dot{X}=&( H_y^+(x,y),-H_x^+(x,y))+\varepsilon (f^+(x,y,\varepsilon, \delta),g^+(x,y,\varepsilon, \delta)),& (x,y)\in \Sigma^{+}\\ \label{p132}
	\dot{X}=& ( H_y^-(x,y),-H_x^-(x,y))+\varepsilon (f^-(x,y,\varepsilon, \delta),g^-(x,y,\varepsilon, \delta)),& (x,y)\in \Sigma^{-}
\end{align}
i.e.,
\begin{align}\label{p133}
\dot{X}=&\begin{cases}
(H_y^+(x,y),-H_x^+(x,y))+\varepsilon (f^+(x,y,\varepsilon, \delta),g^+(x,y,\varepsilon, \delta)),&(x,y)\in \Sigma^{+}\\
( H_y^-(x,y),-H_x^-(x,y))+\varepsilon (f^-(x,y,\varepsilon, \delta),g^-(x,y,\varepsilon, \delta)),&(x,y)\in \Sigma^{-}
\end{cases}.
\end{align}
System (\ref{p133}) is a perturbation of the Hamiltonian system
\begin{align}\label{p134}
\dot{X}=&(H_y^+(x,y),-H_x^+(x,y)),(x,y)\in \Sigma^{+}\\ \label{p135}
\dot{X}=&( H_y^-(x,y),-H_x^-(x,y)),(x,y)\in \Sigma^{-}
\end{align}
or
\begin{align}\label{p136}
\dot{X}=&\begin{cases}
(H_y^+(x,y),-H_x^+(x,y)),&(x,y)\in \Sigma^{+}\\
( H_y^-(x,y),-H_x^-(x,y)),&(x,y)\in \Sigma^{-}
\end{cases}.
\end{align}
Suppose that the switching manifold for (\ref{p133}) and (\ref{p136}) is $\Sigma=\varphi^{-1}(0)$, where $\varphi(x,y)=y$.
Assume that the system (\ref{p136}) has a period annulus around the origin in some open  region $V$. 
 Let $L_+=V\cap \{(x,0):x>0\}$  and $L_-=V\cap \{(x,0):x<0\}$.
 Let $\Gamma^+_r:H^+(x,y)=r,~ y\geq 0$ be a trajectory of (\ref{p134}) which starts at $P(r)=(p(r),0)$ on $L_+$, ends at the point $P_1(r)=(p_1(r),0)$ on $L_-$ with the time of flight $t^+(r)$. Then the Poincare half return map $\mathcal{P}^+:L^+\rightarrow L^-$ is given by
\begin{align*}
\mathcal{P}^+(p(r))=p_1(r).
\end{align*}
Let $\Gamma^-_r:H^-(x,y)=s, y\leq 0 $ be the trajectory of (\ref{p135}) starting at $P_1(r)$ on $L_-$ and ending at the point $P(r)$ with time of flight $t^-(r)$. Therefore the next half return map $\mathcal{P}^-:L^-\rightarrow L^+$ is given by 

\begin{align*}
\mathcal{P}^-(p_1(r))=p(r).
\end{align*}

\begin{figure}
	\centering
	\includegraphics[width=.4\linewidth]{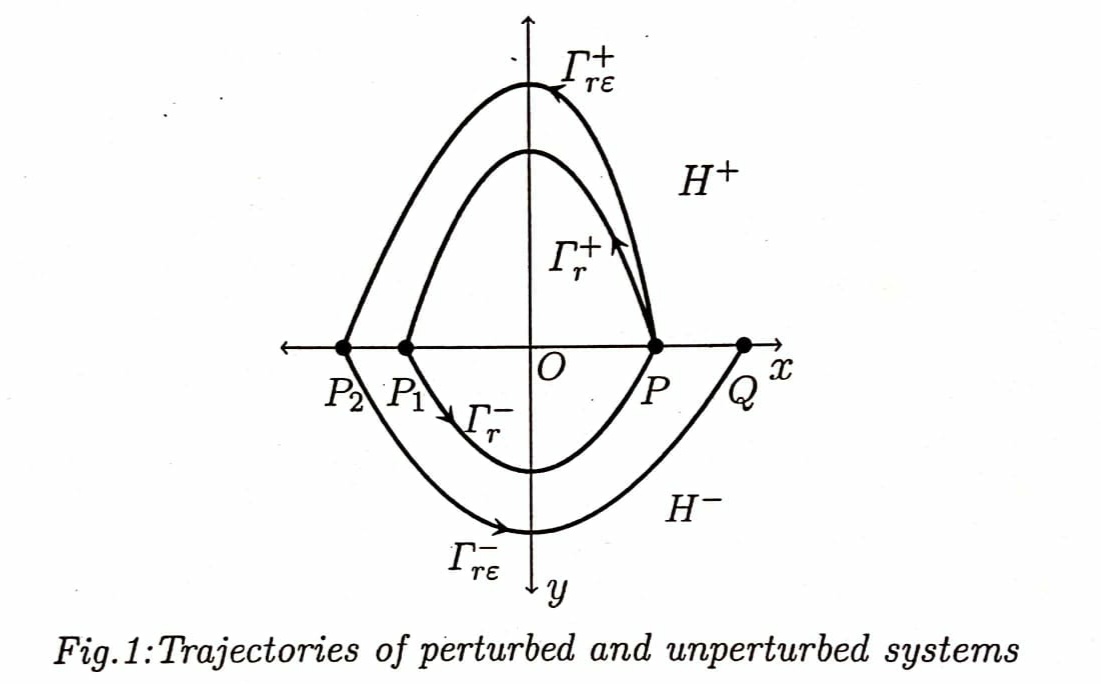}
	\caption{Trajectories showing the difference map}
	\label{fig1}
\end{figure}
Let $\Gamma^+_{r\varepsilon}$ be a trajectory of (\ref{p131}) starting at $P(r)$ and  meeting the first time on $L_-$ at the point $P_2(p_2(r,\varepsilon),0)$ and let $\Gamma^-_{r\varepsilon}$ be the trajectory of the system  (\ref{p132}) starting at $P_2$ and the meeting first time on $L_+$ at the point $Q(q(r,\varepsilon),0)$ (Fig\ref{fig1}).
\par Then the Poincare map $\mathcal{P}_{\varepsilon}$ for (\ref{p133}) defined on  $L_+$ is given by
\begin{align*}
\mathcal{P}_{\varepsilon}(p(r))=q(r,\varepsilon).
\end{align*}
Observe that $\Gamma^+_{r\varepsilon}\cup \Gamma^-_{r\varepsilon}$ forms a closed trajectory of the system (\ref{p133}) if and only if  $p(r)=q(r,\varepsilon)$. But  $p(r)=q(r,\varepsilon)$ is equivalent to $H^+(p(r),0)=H^+(q(r,\varepsilon),0)$. Hence, analogous to the case of smooth differential system, we use the difference map   
\begin{align*}
H^+(Q)-H^+(P)=& \varepsilon F(r,\varepsilon ,\delta)= \sum_{k=1}^{\infty}M_k(r,\delta)\varepsilon ^k,
\end{align*}
where $M_k(r,\delta)$ is called as the $k$th order Melnikov function for system (\ref{p133}) and $F$ is a bifurcation function. 
\par Similar to the Proposition \ref{p12P1} and Proposition \ref{p12P2}, we can state the conditions for cyclicity and stability of limit cycles for the system (\ref{p133}).
\begin{pr}\label{p13P1}
	Assume that the system (\ref{p136}) has a period annulus with center at the origin. Let $M_1(r)$ be the first order Melnikov function for the system (\ref{p133}). Then we have the following:
	\begin{enumerate}
		\item If there exist $0\leq j\leq k$ such that $M_1(r_0,\delta)=0$ and $\dfrac{\partial^j M_1}{\partial r^j}(r_0,\delta)\neq 0$, then at most $k$ limit cycles of (\ref{p133}) are bifurcated form $\Gamma_{r_0}$, where $\Gamma_{r_0}$ is a periodic orbit of (\ref{p136}) through $r_0$.
		\item Limit cycle of (\ref{p133}) bifurcated from the periodic orbit of (\ref{p136}) passing through $(p(r),0)$ of the Poincare section  is stable if and only if $$\dfrac{dM_1}{dr}-M_1(r)\dfrac{H_{xx}^+(p(r),0)}{H_x^+(p(r),0)^2}<0.$$
	\end{enumerate}
\end{pr}
\begin{proof} Proof is similar to that of Proposition \ref{p12P1} and Proposition \ref{p12P2}.\end{proof}
 \par In \cite{liu2010bifurcation}, the first order Melnikov function for (\ref{p133}) when $f^{\pm},g^{\pm}$ are independent of $\varepsilon$ and $\delta$, is given by the following proposition.
\begin{pr}\label{p13P2}\cite{liu2010bifurcation}
	If the system (\ref{p136}) has a period annulus then the first order Melnikov function for the system (\ref{p133}) is 
	\begin{align*}
	M_1(r)=&\dfrac{H^+_x(P)}{H_x^-(P)}\left[\dfrac{H_x^-(P_1)}{H_x^+(P_1)}\int_{\widehat{PP_1}}(g^+dx-f^+dy)+\int_{\widehat{P_1P}}(g^-dx-f^-dy)\right],
	\end{align*} where $\widehat{PP_1}$ denote the path along the trajectory  $\Gamma^+_r$ and $\widehat{P_1P}$ denote the path along the trajectory  $\Gamma^-_r$.
\end{pr}

 In this section, we first derive the expressions for first order and second order Melnikov functions for piecewise smooth perturebed Hamiltonian system
 \begin{align}\label{p137}
 \dot{X}=& ( H_y^+(x,y)+\varepsilon f_1^+(x,y)+\varepsilon^2 f_2^+(x,y),  -H_x^+(x,y)+\varepsilon g_1^+(x,y)+\varepsilon^2 g_2^+(x,y)),~ y>0 \\ \label{p138}
  \dot{X} =&( H_y^-(x,y)+\varepsilon f_1^-(x,y)+\varepsilon^2 f_2^-(x,y),  -H_x^-(x,y)+\varepsilon g_1^-(x,y)+\varepsilon^2 g_2^-(x,y)),~ y<0
 \end{align}
 or
 \begin{align}\label{p139}
 \dot{X}=&\begin{cases}
( H_y^+(x,y)+\varepsilon f_1^+(x,y)+\varepsilon^2 f_2^+(x,y),  -H_x^+(x,y)+\varepsilon g_1^+(x,y)+\varepsilon^2 g_2^+(x,y)),~ y>0 \\
 ( H_y^-(x,y)+\varepsilon f_1^-(x,y)+\varepsilon^2 f_2^-(x,y),  -H_x^-(x,y)+\varepsilon g_1^-(x,y)+\varepsilon^2 g_2^-(x,y)),~ y<0
 \end{cases},
 \end{align} 
 under the assumption that the uperturbed system  (\ref{p136}) has a period annulus around the origin.
 \begin{thm}\label{p13T1}
    If the system (\ref{p136}) has a period annulus around the origin, 
     then the first order Melnikov functions for the system (\ref{p139}) is given by
    \begin{align}\label{p1310}
     M_1(r)=&\dfrac{H_x^+(P)}{H_x^-(P)}\left( \dfrac{H_x^-(P_1)}{H_x^+(P_1)}\int_{\Gamma_r^+}\omega_1^++\int_{\Gamma_r^-}\omega_1^-\right).
    \end{align}
   Further, if $M_1\equiv 0$ then the second order Melnikov function $M_2$ for the system (\ref{p139}) is given by
   \begin{align}\label{p1311}
   M_2(r)\frac{H_x^-(P)}{H_x^+(P)}
   =&\left(\int_{\Gamma_r^-}\omega_2^-+\dfrac{H_x^-(P_1)}{H_x^+(P_1)}\int_{\Gamma_r^+}\omega_2^+ \right)\nonumber\\&+\left(K^-(P(r))\int_{\Gamma_r^-}\frac{\omega_1^-}{H_y^-}+\dfrac{H_x^-(P_1)K^+(P(r))}{H_x^+(P_1)}\int_{\Gamma_r^+}\frac{\omega_1^+}{H_y^+}\right)\nonumber\\
   &-\left(\int_{\Gamma^-_r}\frac{f_1^-\omega_1^-}{H_y^-}+\frac{H_x^-(P_1)}{H_x^+(P_1)}\int_{\Gamma_r^+}\frac{f_1^+\omega_1^+}{H_y^+}\right)+\frac{1}{2}\left( H_{xx}^-(P_1)-\frac{H_x^-(P_1)}{H_x^+(P_1)}H_{xx}^+(P_1)\right)\sigma^2, 
   \end{align}
      
     where 
     \begin{align} 
    & \omega_i^{\pm}= g_i^{\pm}dx-f_i^{\pm}dy~\mbox{for}~i=1,2; ~\sigma=\dfrac{1}{H_x^+(P_1)}\int_{\Gamma_r^+}\omega_1^+,~~
     K^{\pm}=\frac{H_x^{\pm}f_1^{\pm}+H_y^{\pm}g_1^{\pm}}{H_x^{\pm}}.\nonumber
     \end{align}
     
 \end{thm}

We prove Theorem \ref{p13T1} in sequence of following Lemmas. 
\begin{lm}\label{p13L1}
	Difference map for the system (\ref{p139}) can be expressed as,
	\begin{align}\label{p1312}
	H^+(Q(q(r,\varepsilon),0)-H^+(P(p(r),0))=\varepsilon H_x^+(P)\rho+\dfrac{\varepsilon^2}{2}\left[H_{xx}^+(P)\rho^2+H_x^+(P)\eta \right]+o(\varepsilon^3),
	\end{align}
	where $\displaystyle \rho=\left[\frac{\partial }{\partial \varepsilon}\left(q(r,\varepsilon)\right)\right]_{\varepsilon=0}$ and $\displaystyle \eta=\left[\frac{\partial^2 q(r,\varepsilon)}{\partial \varepsilon^2}\right]_{\varepsilon=0}$.
\end{lm}
\begin{proof} Let 
$\displaystyle \sigma=\left[\frac{\partial p_2(r,\varepsilon)}{\partial \varepsilon}\right]_{\varepsilon=0}, \tau=\left[\frac{\partial^2 p_2(r,\varepsilon)}{\partial \varepsilon^2}\right]_{\varepsilon=0}, \rho=\left[\frac{\partial q(r,\varepsilon)}{\partial \varepsilon}\right]_{\varepsilon=0}$ and $\displaystyle \eta=\left[\frac{\partial^2 q(r,\varepsilon)}{\partial \varepsilon^2}\right]_{\varepsilon=0}.$
The difference map for the system (\ref{p139}) is
\begin{align}\label{p1313}
H^+(Q(r,\varepsilon))-H^+(P(r))=&L_1+L_2+L_3+L_4, 
\end{align}
where \begin{align}
&L_1=H^+(Q(r))-H^-(Q(r)), &L_2=H^-(Q(r))-H^-(P_2(r)),\nonumber \\ &L_3=H^-(P_2(r))-H^+(P_2(r)), &L_4=H^+(P_2(r))-H^+(P(r)).\nonumber
\end{align}
Now by the Taylor's series expansion in powers of $\varepsilon$ we have
\begin{align*} 
\sum_{i=1}^{4}L_i(r,\varepsilon)=&\sum_{i=1}^{4}\left[ \varepsilon \left(\frac{\partial L_i}{\partial \varepsilon}\right)_{\varepsilon=0}+ \frac{\varepsilon^2}{2} \left(\frac{\partial^2 L_i}{\partial \varepsilon^2}\right)_{\varepsilon=0}+o(\varepsilon^3) \right].
\end{align*}

\par Since \begin{align*} 
L_1=&H^+(Q)-H^-(Q)= H^+(q(r,\varepsilon),0)-H^-(q(r,\varepsilon),0),\end{align*} we get \begin{align}\label{p1314}
\left[\dfrac{\partial L_1}{\partial \varepsilon}\right]_{\varepsilon=0}=&
(H_x^+(P)-H_x^-(P))\rho, \\ \label{p1315} \mbox{and}~~~~~~
\left[\dfrac{\partial^2 L_1}{\partial \varepsilon^2}\right]_{\varepsilon=0}=&
(H_{xx}^+(P)-H_{xx}^-(P))\rho^2+(H_x^+(P)-H_x^-(P))\eta.
\end{align}
Similarly,
\begin{align}
L_3=&H^-(P_2)-H^+(P_2)= H^-(p_2(r,\varepsilon),0)-H^+(p_2(r,\varepsilon),0)\nonumber
\end{align}
imply that
\begin{align}\label{p1316}
\left[\frac{\partial L_3}{\partial \varepsilon}\right]_{\varepsilon=0}=&
(H_x^-(P_1)-H_x^+(P_1))\sigma, \\   \label{p1317} \mbox{and} ~~
\left[\frac{\partial^2 L_3}{\partial \varepsilon^2}\right]_{\varepsilon=0}=&
(H_{xx}^-(P_1)-H_{xx}^+(P_1))\sigma^2+(H_x^-(P_1)-H_x^+(P_1))\tau .
\end{align}
Also,
\begin{align} L_2=&H^-(Q)-H^-(P_2)=H^-(q(r,\varepsilon),0)-H^-(p_2(r,\varepsilon),0)\nonumber
\end{align} gives us
\begin{align}\label{p1318}
\left[\frac{\partial L_2}{\partial \varepsilon}\right]_{\varepsilon=0}
=&H_x^-(P)\rho-H_x^-(P_1)\sigma,\\   \label{p1319}\mbox{and} ~~
\left[\frac{\partial^2 L_2}{\partial \varepsilon^2}\right]_{\varepsilon=0}
=&H_{xx}^-(P)\rho^2-H_{xx}^-(P_1)\sigma^2+H_x^-(P)\eta-H_x^-(P_1)\tau.
\end{align}
Further,
\begin{align*}
L_4=H^+(P_2)-H^+(P)=H^+((p_2(r,\varepsilon),0))-H^+(p(r),0),
\end{align*} so that
\begin{align}\label{p1320}
\left[\frac{\partial L_4}{\partial \varepsilon}\right]_{\varepsilon=0}
=&H_x^+(P_1)\sigma,\\  \label{p1321}
\mbox{and} ~~~~~
\left[\frac{\partial^2 L_4}{\partial \varepsilon^2}\right]_{\varepsilon=0}
=&H_{xx}^+(P_1)\sigma^2+H_x^+(P_1)\tau.
\end{align}
Hence, from equations (\ref{p1313})-(\ref{p1321}) we get
equation (\ref{p1312}).
\end{proof}
\begin{lm}\label{p13L2}
	The expression for $L_4$ is given by
	\begin{align}\label{p1322}
	L_4=&\varepsilon \int_{\Gamma^+_r}\omega_1^++\varepsilon^2\left(\int_{\Gamma^+_r}\omega_2^++K^+(P(r))\int_{\Gamma^+_r} \frac{\omega_1^+}{H_y^+}-\int_{\Gamma_r^+}\frac{f_1^+\omega_1^+}{H_y^+}\right)+o(\varepsilon^3),
	\end{align}
	where 
	$\displaystyle \omega_i^+=g^+_i dx -f^+_i dy ,i=1,2$.
\end{lm}	
\begin{proof}
\par We have,
\begin{align}\label{p1323} 
L_4=&H^+(P_2)-H^+(P) =\int_{\widehat{PP_2}}dH^+
=\int_{\widehat{PP_2}} H^+_x dx+H^+_ydy \nonumber \\
=&\int_{\widehat{PP_2}} \left[ H^+_x (H_y^++\varepsilon f_1^++\varepsilon^2 f_2^+)+H^+_y(-H_x^++ \varepsilon g_1^+ +\varepsilon^2 g_2^+)\right] dt \nonumber\\
=&\varepsilon\int_{\widehat{PP_2}} ( H^+_x f_1^++H_y^+g_1^+)dt +\varepsilon^2\int_{\widehat{PP_2}}(H_x^+ f_2^++H_y^+g_2^+) dt \nonumber\\ 
=&\varepsilon\int_{\widehat{PP_2}} ( H^+_x f_1^++H_y^+g_1^+)dt +\varepsilon^2\int_{\widehat{PP_1}}(H_x^+ f_2^++H_y^+g_2^+) dt+o(\varepsilon^3)\nonumber\\
=&\varepsilon\int_{\widehat{PP_2}} ( H^+_x f_1^++H_y^+g_1^+)dt +\varepsilon^2\int_{\Gamma_r^+}\omega_2^+ +o(\varepsilon^3).
\end{align} 
\par Along the path $\widehat{PP_2}$, we have $dt=\dfrac{dy}{\dot{y}}=\dfrac{dy}{-H_x^++\varepsilon g_1^++\varepsilon^2 g_2^+}=\dfrac{-1}{H_x^+}\left(1+\varepsilon \dfrac{g_1^+}{H_x^+}+o(\varepsilon ^2)\right)dy.$ 
\par Let us denote $K^+=\dfrac{H_x^+f_1^++H_y^+g_1^+}{H_x^+}, K^+dy=-g_1^+dx+f_1^+dy=-\omega_1^+ $ on $\Gamma_r^+$ and $R$ is the region bounded by $\Gamma^+_{r\varepsilon}$ and  $\Gamma^+_r\cup \overrightarrow{P_1P_2}$, where $\overrightarrow{P_1P_2}$ denote the line segment from $P_1$ to $P_2$. Then
\begin{align}\label{p1324}
\int_{\widehat{PP_2}} ( {H_x^+} {f_1}^+ &+{H_y}^+{g_1}^+)dt
=\int_{\widehat{PP_2}} -\dfrac{{H_x^+} {f_1}^+ +{H_y}^+{g_1}^+}{{H_x}^+}\left(1+\varepsilon \dfrac{{g_1}^+}{{H_x}^+}+o(\varepsilon ^2)\right)dy \nonumber \\
=&-\int_{\widehat{PP_2}}\dfrac{{H_x}^+{f_1}^++{H_y}^+{g_1}^+}{{H_x}^+}dy-\varepsilon \int_{\widehat{PP_2}}\dfrac{{g_1}^+({H_x}^+{f_1}^++{H_y}^+{g_1}^+)}{{{H_x}^+}^2}dy+o(\varepsilon^2)\nonumber\\
=&-\int_{\widehat{PP_2}}K^+dy-\varepsilon \left(\int_{\widehat{PP_1}}\dfrac{{g_1}^+K^+}{{H_x}^+}dy+o(\varepsilon)\right)+o(\varepsilon^2) \nonumber\\
=&-\int_{\widehat{PP_1}}K^+dy-\int_{\overrightarrow{P_1P_2}}K^+dy-\iint_R \dfrac{\partial K^+}{\partial x}dxdy-\varepsilon \left(\int_{\widehat{PP_1}}\dfrac{K^+{g_1}^+}{{H_x}^+}dy+o(\varepsilon)\right)+o(\varepsilon^2) \nonumber\\
=&\int_{\Gamma^+_r}\omega^+_1-\int_{\overrightarrow{P_1P_2}}K^+dy-\iint_R \dfrac{\partial K^+}{\partial x}dxdy-\varepsilon \int_{\Gamma^+_r}\dfrac{K^+{g_1}^+}{{H_x}^+}dy+o(\varepsilon^2)\nonumber\\
=&\int_{\Gamma^+_r}\omega^+_1-\iint_R \dfrac{\partial K^+}{\partial x}dxdy+\varepsilon \int_{\Gamma^+_r}\dfrac{{g_1}^+\omega_1^+}{{H_x}^+}+o(\varepsilon^2).
\end{align}

Now suppose that $R=R_1\cup R_2$, where $R_1$ is the region bounded by $\Gamma^+_r,\Gamma^+_{r\varepsilon},x=p_1(r)$ and $x=p(r)$ whereas $R_2$ is bounded by $y=0,\Gamma^+_{r\varepsilon},x=p_2(r,\varepsilon)$ and $x=p_1(r)$. Note that, since the radial distance $z$ form $P_1(r)$ to the point on $\Gamma^+_{r\varepsilon}$ in $R_2$ is of order 
$\varepsilon$, we have $\displaystyle \int\int_{R_2}dxdy=\int_{0}^{\frac{\pi}{2}}\int_{0}^{z}rdrd\theta=\frac{\pi}{4}z^2=o(\varepsilon^2)$. Therefore 
\begin{align}\label{p1325}
\displaystyle \iint_{R_2}\dfrac{\partial K^+}{\partial x}dxdy=o(\varepsilon^2).
\end{align} 
\par Let us represent $\Gamma_{r\varepsilon}^+$ and $\Gamma_r^+$  by $ y_{\varepsilon}=y(x,\varepsilon),y_0=y(x,0) $ respectively.
Put $y(x,s)=y_0(x)+s(y_{\varepsilon}(x)-y_0(x)),~~p_1(r)\leq x\leq p(r),~~0\leq s\leq 1.$ Hence, area element for region $R_1$ becomes $dydx=(y_{\varepsilon}(x)-y_0(x))dsdx$. Therefore
\begin{align}\label{p1326}
\iint_{R_1}\dfrac{\partial K^+}{\partial x}dxdy 
=&\int_{p_1(r)}^{p(r)}\left(\int_{0}^{1}\dfrac{\partial K^+}{\partial x}(y_{\varepsilon}(x)-y_0(x))ds\right)dx \nonumber\\
=&\varepsilon \int_{p_1(r)}^{p(r)}\left(\int_{0}^{1}\dfrac{\partial K^+}{\partial x}\left(\frac{\partial y_{\varepsilon}}{\partial \varepsilon}\right)_{\varepsilon=0}ds\right)dx +o(\varepsilon^2)\nonumber\\
=&\varepsilon \int_{p_1(r)}^{p(r)}\left(\int_{0}^{1}\dfrac{\partial K^+}{\partial x}(x,y_0(x))\left(\frac{\partial y_{\varepsilon}}{\partial \varepsilon}\right)_{\varepsilon=0}ds\right)dx +o(\varepsilon^2)\nonumber\\ 
=& \varepsilon \int_{p_1(r)}^{p(r)}\dfrac{\partial K^+}{\partial x}(x,y_0)\left(\frac{\partial y_{\varepsilon}}{\partial \varepsilon}\right)_{\varepsilon=0}dx +o(\varepsilon^2).
\end{align}
Now along  $\Gamma^+_{r\varepsilon}$, we have 
\begin{align}
&\dot{ y_{\varepsilon}}=\dfrac{\partial y_{\varepsilon}}{\partial x}\dot{x}=-{H_x}^+ +\varepsilon {g_1}^+ +\varepsilon^2 {g_2}^+ \Rightarrow  \dfrac{\partial y_{\varepsilon}}{\partial x}=\dfrac{-{H_x}^+ +\varepsilon {g_1}^+ +\varepsilon^2 {g_2}^+}{{H_y}^+ +\varepsilon {f_1}^+ +\varepsilon^2 {f_2}^+}\nonumber
 \\
\Rightarrow & y_{\varepsilon}=\int_{p_1 (r)}^x\dfrac{-{H_x}^+ +\varepsilon {g_1}^+ +\varepsilon^2 {g_2}^+}{{H_y}^+ +\varepsilon {f_1}^++\varepsilon^2 {f_2}^+}ds =\int_{p_1(r)}^x \dfrac{-{H_x}^+}{{H_y}^+}ds+\varepsilon \int_{p_1(r)}^x \dfrac{{H_x}^+{f_1}^+ +{H_y}^+{g_1}^+}{{{H_y}^+}^2}ds+o(\varepsilon^2).\nonumber
\end{align}
Note here that last integrals is taken along $\Gamma_r^+$. From above expression we have
\begin{align}\label{p1327}
\left(\dfrac{\partial y_{\varepsilon}}{\partial\varepsilon}\right)_{\varepsilon =0}=&\int_{p_1(r)}^x \dfrac{{H_x}^+{f_1}^+ +{H_y}^+{g_1}^+}{{{H_y}^+}^2}ds 
=\int_{p_1(r)}^x\dfrac{K^+{H_x}^+}{({H_y}^+)^2}dx=I^+(x), ~\mbox{(say)}.
\end{align} 
Therefore, from equations (\ref{p1325}), (\ref{p1326}) and (\ref{p1327}), we get
\begin{align*}
\iint_{R}\dfrac{\partial K^+}{\partial x}dxdy=\varepsilon \int_{p_1(r)}^{p(r)}\dfrac{\partial K^+}{\partial x}(x,y_0)I^+dx+o(\varepsilon^2).
\end{align*}
Now using integration by parts, we have
\begin{align}\label{p1328}
\iint_{R}\dfrac{\partial K^+}{\partial x}dxdy=\varepsilon\left( K^+(P(r))\int_{p_1(r)}^{p(r)}\frac{K^+H_x^+}{(H_y^+)^2}dx\right)-\varepsilon \int_{p_1(r)}^{p(r)}\frac{(K^+)^2H_x^+}{(H_y^+)^2}dx+o(\varepsilon^2)\nonumber\\
=-\varepsilon K^+(P(r))\int_{\Gamma_r^+}\frac{\omega_1^+}{H_y^+}+\varepsilon \int_{\Gamma_r^+}\frac{f_1^+\omega_1^+}{H_y^+}+\varepsilon \int_{\Gamma_r^+}\frac{g_1^+\omega_1^+}{H_x^+}+o(\varepsilon^2).
\end{align}
Hence, from (\ref{p1323}), (\ref{p1324}) and (\ref{p1328}) we obtain the formula for $L_4$.
\end{proof}
\begin{lm}\label{p13L3}
The expression for $L_2$ is given by
	\begin{align}\label{p1329}
	L_2=&\varepsilon \int_{\Gamma^-_r}\omega_1^-+\varepsilon^2\left(\int_{\Gamma^-_r}\omega_2^-+K^-(P(r))\int_{\Gamma^-_r}\frac{\omega_1^-}{H_y^-}-\int_{\Gamma_r^-}\frac{f_1^-\omega_1^-}{H_y^-}\right)+o(\varepsilon^3),
	\end{align}
	where 
	$\displaystyle \omega_i^-=g^-_i dx -f^-_i dy ,i=1,2 $.
\end{lm}
\begin{proof}
	Proof is similar to the proof of Lemma \ref{p13L2}.
\end{proof}

\begin{proof} [Proof of Theorem \ref{p13T1}]
\par Comparing coefficients of $\varepsilon$ and $\varepsilon^2$ in the expression for $L_4$ obtained in Lemma (\ref{p13L2}) and from expressions (\ref{p1320}), (\ref{p1321}), we have
\begin{align}\label{p1330}
\sigma=\dfrac{1}{H_x^+(P_1)}\int_{\Gamma_r^+}\omega_1^+,~~ \mbox{and}
\end{align} 
\begin{align}\label{p1331}
\dfrac{1}{2}H_x^+(P_1)\tau=&\int_{\Gamma^+_r}\omega_2^++K^+(P(r))\int_{\Gamma^+_r} \frac{\omega_1^+}{H_y^+}-\int_{\Gamma_r^+}\frac{f_1^+\omega_1^+}{H_y^+} -\dfrac{1}{2}H_{xx}^+(P_1)\sigma^2.
\end{align} 
Similarly, from Lemma (\ref{p13L3}) and expressions (\ref{p1318}), (\ref{p1319}) we get
\begin{align}\label{p1332}
H_x^-(P)\rho=&H_x^-(P_1)\sigma+ \int_{\Gamma_r^-}\omega_1^-=\dfrac{H_x^-(P_1)}{H_x^+(P_1)}\int_{\Gamma_r^+}\omega_1^++\int_{\Gamma_r^-}\omega_1^-,~~\mbox{and}
\end{align}
\begin{align} \label{p1333}
\frac{1}{2}H_x^-(P)\eta=&\int_{\Gamma^-_r}\omega_2^-+K^-(P(r))\int_{\Gamma^-_r}\frac{\omega_1^-}{H_y^-}-\int_{\Gamma_r^-}\frac{f_1^-\omega_1^-}{H_y^-}-\frac{1}{2}H_{xx}^-(P)\rho^2+\frac{1}{2}H_{xx}^-(P_1)\sigma^2+\frac{1}{2}H_x^-(P_1)\tau.
\end{align}
From (\ref{p1312}) and (\ref{p1332}) we get the first order Melnikov function.
\par Now if the first order Melnikov function is identically zero, then from equation (\ref{p1312}) we have  $\rho\equiv 0$. Hence, from (\ref{p1312}) and (\ref{p1333}), the second order Melnikov function is  
\begin{align*}
M_2(r)={H_x}^+(P)\eta 
=\dfrac{{H_x}^+(P)}{{H_x}^-(P)}\left(\int_{\Gamma^-_r}{\omega_2}^--\int_{\Gamma^-_r}\frac{{f_1}^-{\omega_1}^-}{{H_y}^-}+K^-(P(r))\int_{\Gamma_r^-}\frac{{\omega_1}^-}{{H_y}^-}+\frac{1}{2}{H_{xx}}^-(P_1)\sigma^2+\frac{1}{2}{H_x}^-(P_1)\tau \right).\end{align*} 
and hence by substituting $\tau$ from (\ref{p1331}) we get the required expression for $M_2$.
\end{proof}
\par If the Hamiltonian system (\ref{p136}) is extended smoothly on the boundary $y=0$, then it becomes a smooth Hamiltonian system. In this case the first order and second order  Melnikov function are simple line integrals of one forms. Further, if perturbation of this system is also smooth, then the first order and second order Melnikov functions obtained from Theorem \ref{p13T1} are well known integrals of one forms as in the following corollary.  
\begin{cor}\label{p13C1}
	If  in the  system (\ref{p139}) $\displaystyle \lim_{y\rightarrow 0^+} {H_x}^+(x,y)=\lim_{y\rightarrow 0^-}{H_x}^-(x,y)$ and $\displaystyle \lim_{y\rightarrow 0^+}{H_y}^+(x,y)=\lim_{y\rightarrow 0^-}{H_y}^-(x,y)$ for all $x\in \R$, then the first order  and the second order Melnikov functions are given by
	\begin{align}\label{p1334}
	M_1(r)&=\int_{\Gamma_r^+}\omega_1^++\int_{\Gamma_r^-}\omega_1^- ~~~\mbox{and}\\\label{p1335}
	M_2(r) 
	&=\left(\int_{\Gamma_r^-}\omega_2^-+\int_{\Gamma_r^+}\omega_2^+ \right)-\left(\int_{\Gamma^-_r}\frac{f_1^-\omega_1^-}{H_y^-}+\int_{\Gamma_r^+}\frac{f_1^+\omega_1^+}{H_y^+}\right)\nonumber\\
	&+\left(K^-(P(r))\int_{\Gamma_r^-}\frac{\omega_1^-}{H_y^-}+K^+(P(r))\int_{\Gamma_r^+}\frac{\omega_1^+}{H_y^+}\right),
	\end{align} respectively.
	\par Further, if $\displaystyle \lim_{y\rightarrow 0^+} f_i^+(x,y)=\lim_{y\rightarrow 0^-}f_i^-(x,y)$ and $\displaystyle \lim_{y\rightarrow 0^+}g_i^+(x,y)=\lim_{y\rightarrow 0^-}g_i^-(x,y)$ for $ i=1,2$ and for all $x\in \R$, then the
	first order and second order Melnikov functions are given by
	\begin{align}\label{p1336}
	M_1(r)=\oint_{\Gamma_r}\omega_1,~~\mbox{and}~~ M_2(r)=\oint_{\Gamma_r}\omega_2-\oint_{\Gamma_r}\frac{f_1\omega_1}{H_y}+K(P(r))\oint_{\Gamma_r}\frac{\omega_1}{H_y},
	\end{align} respectively,\\
	where $\displaystyle H_x(x,y)=\begin{cases}
	H_x^+(x,y)& \mbox{if}~y>0\\ H_x^-(x,y)&\mbox{if}~y<0 \\ \displaystyle \lim_{y\rightarrow 0^+}H_x^+(x,y)&\mbox{if}~y=0
	\end{cases},~ H_y(x,y)=\begin{cases}
	H_y^+(x,y)& \mbox{if}~y>0\\ H_y^-(x,y)&\mbox{if}~y<0 \\ \displaystyle \lim_{y\rightarrow 0^+}H_y^+(x,y)&\mbox{if}~y=0
	\end{cases},\\ f_i(x,y)=\begin{cases}
	f_i^+(x,y)& \mbox{if}~y>0\\ f_i^-(x,y)&\mbox{if}~y<0 \\ \displaystyle \lim_{y\rightarrow 0^+}f_i^+(x,y)&\mbox{if}~y=0
	\end{cases},~ g_i(x,y)=\begin{cases}
	g_i^+(x,y)& \mbox{if}~y>0\\ g_i^-(x,y)&\mbox{if}~y<0 \\ \displaystyle \lim_{y\rightarrow 0^+}g_i^+(x,y)&\mbox{if}~y=0
	\end{cases}, \\ \omega_i=g_idx-f_idy,~~K=\frac{H_xf_1+H_yg_1}{H_x}$ for $i=1,2$ and $\Gamma_r=\Gamma_r^+\cup \Gamma_r^-$, a closed trajectory of the unperturbed system.
\end{cor}
\begin{proof} Proof follows from the equation (\ref{p1310}) and (\ref{p1311}).
	\end{proof}
\section{Piecewise Hamiltonian System with Boundary Perturbation}
Now consider a piecewise Hamiltonian system with boundary perturbation,
\begin{align}\label{p141}
\dot{X}=&\begin{cases}
(H_y^+,-H_x^+), ~y>\varepsilon f(x)\\
(H_y^-,-H_x^-),~ y<\varepsilon f(x)
\end{cases},
\end{align}
where $H^+,H^-:\R^2\rightarrow \R$ are $C^2$ functions and $f:\R \rightarrow \R $ is a $C^1$ function.
\par Here $\Sigma =\{(x,y)\in \mathbb{R}^2: y=\varepsilon f(x)\}$ is a switching manifold and $\Sigma^{\pm}=\{(x,y)\in \mathbb{R}^2: \pm(y-\varepsilon f(x))> 0\}$ are two zones saparated by $\Sigma$.
\begin{remark}
According to the Filippov convension \cite{kuznetsov2003one}, the switching manifold $\Sigma$ is divided into the following regions:
\par Crossing region  $ \Sigma_c=\{(x,\varepsilon f(x)):H_x^{\pm}+\varepsilon H_y^{\pm}f'(x)>0~\mbox{or}~ H_x^{\pm}+\varepsilon H_y^{\pm}f'(x)<0\}$, \par Sliding region $\Sigma_s=\{(x,\varepsilon f(x)):H_x^{+}+\varepsilon H_y^{+}f'(x)<0,H_x^{-}+\varepsilon H_y^{-}f'(x)>0\}$, and \par Escaping region $\Sigma_e=\{(x,\varepsilon f(x)):H_x^{+}+\varepsilon H_y^{+}f'(x)>0,H_x^{-}+\varepsilon H_y^{-}f'(x)<0\}$. 
\par \noindent Discontinuity induced bifurcations are studied according to these regions.
 \end{remark}
\par Using the analytic invertible change of variables $u =x, v =y-\varepsilon f(x)$ and renaming the variables $u$ by $x$ and $v$ by $y$, the system (\ref{p141}) becomes
 \begin{align}\label{p142}
 \dot{X}=& \begin{cases}
 (X^+,Y^+), ~y>0\\
 (X^-,Y^-), ~y<0
 \end{cases},
 \end{align}
 where
 \begin{align}
 X^+=& H_y^++\varepsilon f(x)H^+_{yy}+\varepsilon ^2 \dfrac{1}{2}f(x)^2H^+_{yyy}+o(\varepsilon ^3),\nonumber \\
Y^+=& -H^+_x-\varepsilon (f'(x)H^+_y+f(x)H^+_{xy}) -\varepsilon^2 f(x)\left(\dfrac{1}{2}f(x)H^+_{xyy}+f'(x)H^+_{yy}\right)+o(\varepsilon^3),\nonumber \\
  X^-=&H_y^-+\varepsilon f(x)H^-_{yy}+\varepsilon ^2 \dfrac{1}{2}f(x)^2H^-_{yyy}+o(\varepsilon ^3),\nonumber\\
Y^-=& -H^-_x-\varepsilon (f'(x)H^-_y+f(x)H^-_{xy}) -\varepsilon^2 f(x)\left(\dfrac{1}{2}f(x)H^-_{xyy}+f'(x)H^-_{yy}\right)+o(\varepsilon^3).\nonumber
 \end{align} 
 \begin{remark}
 	Note that if $\Phi(x,y)=(x,y-\varepsilon f(x))$, then $\Phi$ is diffeormorphism. Therefore, systems (\ref{p141}) and (\ref{p142}) are topologically equivalent and their flows are $C^1$ conjugate.\end{remark} 
Melnikov function for the system (\ref{p142}) are given by the following theorem.
\begin{thm} If the system (\ref{p142}) at $\varepsilon =0$ has a period annulus around the origin and   $H_y^{\pm}(P)=H_y^{\pm}(P_1)$, then  the first order Melnikov function for (\ref{p142}) is 
	 \begin{align}\label{p143}
	\lambda M_1(r)={H_x^-(P_1)}\left(\dfrac{H_y^+(P_1)}{H_x^+(P_1)}-\dfrac{H_y^-(P_1)}{H_x^-(P_1)}\right)(f(p)-f(p_1)).
	\end{align} and the second order Melnikov function is given by  \begin{align}\label{p144}
	\lambda M_2(r)=&\dfrac{H_x^-(P_1)}{2}\left(\dfrac{H_{yy}^+(P_1)}{H_x^+(P_1)}-\dfrac{H_{yy}^-(P_1)}{H_x^-(P_1)}\right)((f(p(r)))^2-(f(p_1(r)))^2)\nonumber\\
	&+\dfrac{H_x^-(P_1)K^+(P)-H_x^+(P_1)K^-(P)}{H_x^+(P_1)}\left(f(p(r))-f(p_1(r)) \right)\nonumber \\
	&-\left(K^-(P)\int_{\Gamma^-_r}\dfrac{f}{H_y^-}d(H_y^-)+\dfrac{H_x^-(P_1)K^+(P)}{H_x^+(P_1)}\int_{\Gamma_r^+}\dfrac{f}{H_y^+}d(H_y^+)\right)\nonumber \\
	&+\int_{\Gamma_r^-}H_{yy}^-\left[d\left(\frac{f^2}{2}\right)+\left(\frac{f^2}{{H_y}^-}\right)d(H_y^-)\right]+\dfrac{H_x^-(P_1)}{H_x^+(P_1)}\int_{\Gamma_r^+}H_{yy}^+\left[d\left(\frac{f^2}{2}\right)+\left(\frac{f^2}{H_y^+}\right)d(H_y^+)\right]\nonumber \\
	&+\dfrac{1}{2}\left(H_{xx}^-(P_1)-\dfrac{H_x^-(P_1)}{H_x^+(P_1)}H_{xx}^+(P_1)\right)\sigma^2,
	\end{align}
	where $\lambda=\dfrac{H_x^-(P)}{H_x^+(P)}.$
\end{thm}
  \begin{proof} 
 From (\ref{p142}), we have
  \begin{align}
 \int_{\Gamma^+_r}\omega_1^+=&\int_{\Gamma^+_r}(g_1^+dx-f_1^+dy)=-\int_{\Gamma^+_r}d(fH_y^+)=f(p)H_y^+(P)-f(p_1)H_y^+(P_1),~\mbox{and}\nonumber\\
 \int_{\Gamma^-_r}\omega_1^-=&\int_{\Gamma^-(r)}(g_1^-dx-f_1^-dy)=-\int_{\Gamma_r^-}d(fH_y^-)=f(p_1)H_y^-(P_1)-f(p)H_y^-(P).\nonumber
 \end{align} 
 Hence, by Theorem \ref{p13T1},
 \begin{align}\label{p145}
 M_1(r)=&\dfrac{H_x^+(P)}{H_x^-(P)}\left( f(p)\left(\dfrac{ H_x^-(P_1)H_y^+(P)}{H_x^+(P_1)}-H_y^-(P)\right)-f(p_1)\left( \dfrac{ H_x^-(P_1)H_y^+(P_1)}{H_x^+(P_1)}-H_y^-(P_1)\right)\
 \right).
 \end{align}
 Now if $H_y^+(P)=H_y^+(P_1)$ and $H_y^-(P)=H_y^-(P_1)$, then 
 we get (\ref{p143}).
 \par Again, from (\ref{p142}) we have
 $ \omega_1^{\pm}=-d(f(x)H_y^{\pm}),\omega_2^{\pm}=-d\left(\dfrac{f^2(x)H_{yy}^{\pm}}{2}\right).$
 Hence, the formula for $M_2$ follows from (\ref{p1311}).\end{proof}
 \begin{remark}
 If the system (\ref{p142}) at $\varepsilon =0$ is smooth, then $\dfrac{H_y^+(P_1)}{H_x^+(P_1)}-\dfrac{H_y^-(P_1)}{H_x^-(P_1)}=0$. 
 Thus, $M_1(r)=0$ if and only if $f(p(r))=f(p_1(r))$ or the system is smooth. If $f(p(r))=f(p_1(r))$, the periodic orbit passing through $P(r)$ of (\ref{p142}) at $\varepsilon =0$ is also a periodic orbit of (\ref{p142}). Hence, the number of periodic orbits persists under perturbation equals to the number of roots of $f(p(r))-f(p_1(r))$.
 \end{remark} 
 If the system (\ref{p141}) at $\varepsilon =0$ is smoothly extended on $y=0$, then its period annulus persists under perturbation of switching boundary. In the following corollary we obtain the first order and second order Melnikov functions for such system. 
 \begin{cor}\label{p14C2}
 	If the system (\ref{p141}) at $\varepsilon =0$ is smooth then its period annulus persists under smooth perturbation of switching manifold $y=0$.
 \end{cor}
\begin{proof} From (\ref{p143}) we have $\displaystyle M_1(r)=\oint_{\Gamma_r}\omega_1=\oint_{\Gamma_r}-d(f(x)H_y)=0.$
 \par Now from (\ref{p144}) we have	\begin{align}\label{p146} \lambda M_2(r)=-K(P)\oint_{\Gamma_r}\frac{f}{H_y}dH_y+\oint_{\Gamma_r}H_{yy}\frac{f^2}{H_y}dH_y+\oint_{\Gamma_r}H_{yy}d\left(\frac{f^2}{2}\right),
 	\end{align}
 	where $f, K, H_y, H_{yy}$ and $\Gamma_r$ are as defined in Corollary \ref{p13C1}.
 \par Along $\Gamma_r$ we have \begin{align} \label{p147} \dfrac{f(x)}{H_y}dH_y=&\dfrac{f(x)}{H_y}(H_{yx}dx+H_{yy}dy)=\dfrac{f(x)}{H_y}(-H_{yx}\dfrac{H_y}{H_x}+H_{yy})dy\nonumber \\ 
 =&f(x)\left(\dfrac{-H_{xy}}{H_x}+\dfrac{H_{yy}}{H_y}\right)dy=\dfrac{\partial}{\partial y}\left(f(x)\log \left(\frac{H_y}{H_x}\right)\right)dy,
 \end{align}
 so that $\displaystyle \oint_{\Gamma_r}\frac{f}{H_y}dH_y=0. $
Also,
 \begin{align*}
\oint H_{yy}\left(f^2\frac{dH_y}{H_y}+d\left(\frac{f^2}{2}\right)\right)
=\oint H_{yy}\left(f^2\frac{\partial}{\partial y}\left(\log \frac{H_y}{H_x}\right)-ff'\frac{H_y}{H_x}\right)dy
=\oint  H_{yy}\left(f^2v_y-ff'u\right)dy,  
\end{align*}
where $u=\frac{H_y}{H_x}~~\mbox{and}~~v=\log(u)$.
Integrating by parts twice, we get
\begin{align}\label{p148}
&\oint H_{yy}\left(f^2\frac{dH_y}{H_y}+d\left(\frac{f^2}{2}\right)\right)
=\oint  H\left( f^2v_{yyy}-ff'u_yy\right)dy
=r\oint \left( (f^2v)_{yyy}-(ff'u)_{yy}\right)dy=0.
\end{align}
Hence from equation (\ref{p146}), (\ref{p147}) and (\ref{p148}), we get $M_2\equiv 0$.
Thus, we conclude that, no limit cycle is bifurcated from the period annulus of (\ref{p141}).\end{proof}
  \begin{remark}
  	In the above proof we consider an extension of the natural logarithmic function on $\R \cup \{+\infty,-\infty\} $ as, $\log (x)= \begin{cases}
  \ln x, x>0\\ \ln(-x), x<0\\ -\infty , x=0\\ \infty , x=\pm \infty
  	\end{cases}.$ This function is an antiderivative of the functions $f_1(x)=\begin{cases}\dfrac{1}{x}, x>0 \\ \infty  ,x=0\end{cases}$ on $[0,\infty]$ and $f_2(x)=\begin{cases}\dfrac{1}{x}, x<0 \\ -\infty  ,x=0\end{cases}$ on $[-\infty,0]$.
  	\end{remark}
  
   \section{Applications}
   There are various types of planar piecewise smooth systems according to the types of singularities in zones separated by the switching line with center at the origin viz. center-center, saddle-center, center-focus. Here we discuss the limit cycle bifurcation from period annulus due to perturbation of the switching manifold of center- center and saddle-center type using the first and second order Melnikov functions. 
   \subsection{Boundary perturbation of center-center type system}
   Consider the piecewise Hamiltonian system \begin{align}\label{p1500}
   \dot{X}=\begin{cases}
   (-1,2x)& \mbox{if}~ y>0\\ (1,2x)& \mbox{if}~ y<0
   \end{cases}.
   \end{align}
   System (\ref{p1500}) has center at the origin (Fig.\ref{fig51}). Hamiltonian of the system $\dot{X}=(-1,2x)$ is $H^+(x,y)=-y-x^2$. Trajectories of this system at level $h=r^2$ are given by $y+x^2=-r^2$. Hamiltonian for $\dot{X}=(1,2x)$ is $H^-(x,y)=y-x^2$ and its trajectories at the levels $h=r^2$ are given by $y-x^2=r^2.$
   \par Now consider the perturbed piecewise smooth system
    \begin{align}\label{p1501}
   \dot{X}=\begin{cases}
   (-1+\varepsilon f_1^+(x,y)+\varepsilon^2f_2^+(x,y),2x+\varepsilon g_1^+(x,y)+\varepsilon^2g_2^+(x,y))& \mbox{if} ~y>0\\ (1+\varepsilon f_1^-(x,y)+\varepsilon^2f_2^-(x,y),2x+\varepsilon g_1^-(x,y)+\varepsilon^2g_2^-(x,y))& \mbox{if}~y<0,
   \end{cases}
   \end{align} 
   \begin{align*}
  \mbox{where}& &f_1^{+}(x,y)=ax^2+bxy+cy^2,&~~~ f_2^{+}(x,y)=Ax^2+Bxy+Cy^2,\\ & &g_1^{+}(x,y)=dx^2+exy+fy^2,& ~~~ g_2^{-}(x,y)=Dx^2+Exy+Fy^2, \\
  & & f_1^{-}(x,y)=px^2+qxy+s y^2,&~~~ f_2^{-}(x,y)=Px^2+Qxy+Sy^2,\\ 
  & &g_1^{-}(x,y)=lx^2+mxy+ny^2,&~~~ g_2^{-}(x,y)=Lx^2+Mxy+Ny^2.
   \end{align*}
 and $a,b,c,d,e,f,p,q,s,l,m,n,A,B,C,D,E,F,P,Q,S,L,M$ and $N$ are real constants. 
   \par From (\ref{p1310}) we get \begin{align}\label{p1502}
  M_1(r)=r^3\left( \dfrac{8}{15}\left(4b-7f+7n-4q \right) {r}^{2}+ \dfrac{2}{3}\left( l-d \right) \right)
  .\end{align}
  Note that, $M_1(r)\equiv 0$ if \begin{align}\label{p1503}
    4b-7f+7n-4q =0,~\mbox{and}~ d-l=0.\end{align}
    Also, from (\ref{p1310}) the second order Melnikov function is
    \begin{align}\label{p1504}
   M_2(r)=& \left( -{\dfrac {2656\,sn}{315}}+{\dfrac {3712\,qs}{315}}+{\dfrac {3712
   		\,bc}{315}}-{\frac {2656\,cf}{315}} \right) {r}^{9}\nonumber \\
   	&+ \left( \dfrac{8}{15}(p
   \left( 7\,n-4\,q \right)-a \left( 4\,b-7\,f \right) ) - \dfrac {184}{105}(ls+qm+np+af+be+cd)+\dfrac {96}{35}(pq+ab) \right) {r}^{7}\nonumber \\
   &- \dfrac{4}{15} \left( l \left( 7\,n-4\,q \right) +d \left( 4\,b-7\,f \right) 
   \right) {r}^{6}+ \dfrac{4}{15}\left( 14(N-F)+8(B-Q)+lp+ad \right) {r}^{5}\nonumber \\
   &+\dfrac{1}{3} \left( d^2-l^2 \right) {r
   }^{4}+\dfrac{2}{3} \left( L-D \right) {r}^{3}.
   \end{align}
   In the view of the conditions (\ref{p1503}), the expression (\ref{p1504}) becomes $M_2(r)=r^3M_2'(h)$, where $h=r^2$ and
   \begin{align}\label{p1505}
   M_2'(h)=&\dfrac{32}{315}\left(-29(7f-4b)(s+c)+120(sn+cf)\right)h^3\nonumber \\
   &+\left(\dfrac{8}{15}(7f-4b)(p+a)-\dfrac{184}{105}(d(s+c)+qm+np+af+be)+\dfrac{96}{35}(pq+ab)\right)h^2\nonumber \\
   &+\dfrac{4}{15}\left( 14(N-F)+8(B-Q)+d(p+a) \right)h+\dfrac{2}{3} \left( L-D \right).
   \end{align}
   Thus under the conditions (\ref{p1503}), the system (\ref{p1501}) can have at most three limit cycles.
   \par In particular, if $a=p,b=q,c=s,m=-e, f=n=2b,L=-D, F=N, B=Q$, then  (\ref{p1505})  becomes
   \begin{align}\label{p1506}
   M'_2(h)=&-\dfrac{640}{63}bch^3+\left(\dfrac{80}{21}ab+\dfrac{16}{3}b^2-\dfrac{368}{105}cd\right)h^2+\dfrac{8}{15}adh-\dfrac{4}{3}D.   \end{align}
   We can choose constants $a,b,c,d$ and $D$ such that (\ref{p1506}) will have three distinct positive roots.
  In particular, if $a=1, bc=1,d=-\dfrac{8800}{42}$ and $D=-\dfrac{960}{21}$, then (\ref{p1506}) becomes
  \begin{align}\label{p15060}
  M_2'(h)=-\dfrac{640}{63}(h^3-6.000000005h^2+11h-6).
  \end{align} Polynomial (\ref{p15060}) has three positive zeros; $h= 1.000000003, 1.999999980, 3.000000022$.  Consequently, the corresponding system will have three limit cycles (Fig.\ref{fig53}).

  \par Further, in (\ref{p1505}), if $4b=7f, e=m=c=0,a=p=s=n=1,q=-b, d=1/10, f=98.9, L=D, 14(N-F)+8(B-Q)=0$, then 
   \begin{align}\label{p1507}
   M'_2(h)=\dfrac{256}{21}h^3-\dfrac{3680}{21}h^2+\dfrac{4}{75}h.
   \end{align} 
    Note that $(\ref{p1507})$ has two real positive zeros, $h = \dfrac{115}{16}+\dfrac{3}{80}\sqrt{36733}, \dfrac{115}{16}-\dfrac{3}{80}\sqrt{36733}$. Therefore the corresponding system will have two limit cycles (Fig.\ref{fig52}).

  \begin{figure}
  	\centering
  	\begin{subfigure}{.4\textwidth}
  		\includegraphics[width=.4\linewidth]{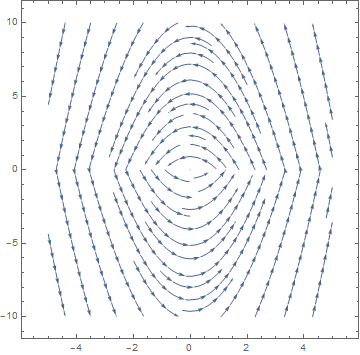}
  		\caption{System (\ref{p1500})}
  		\label{fig51}
  	\end{subfigure}%
  	\begin{subfigure}{.4\textwidth}
  		\includegraphics[width=.4\linewidth]{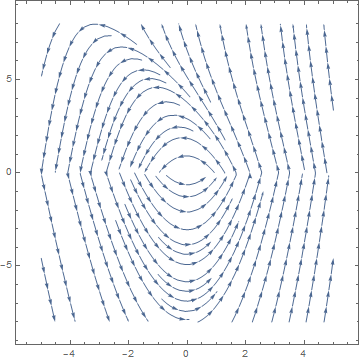}
  		\caption{Two limit cycles}
  		\label{fig52}
  	\end{subfigure}%
  	\begin{subfigure}{.5\textwidth}
  	\includegraphics[width=.4\linewidth]{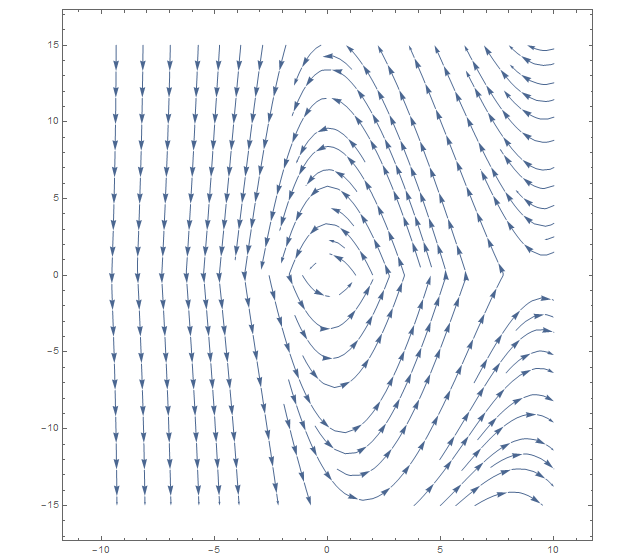}
  	\caption{Three limit cycles}
  	\label{fig53}
  \end{subfigure}%
  	\caption{Flow of the system (\ref{p1501})}
  	\label{fig523}
  \end{figure}

   \par In \cite{wei2016normal}, authers characterize all planar piecewise smooth differential systems having $(k,l)$ center at the origin. Here we mention the result;
   \begin{pr}\label{p15P1}\cite{wei2016normal} Let $k,l,r$ be positive integers and $\max\{k,l \}\leq r\in \N\cup \{\infty, \omega \} $.
   	Suppose that the system \begin{align}\label{p151}
   	\dot{X}=\begin{cases}
   	(F^+(x,y),G^+(x,y))& \mbox{if}~y>0\\ (F^-(x,y),G^-(x,y))& \mbox{if}~y<0
   	\end{cases}\end{align} is piecewise smooth with $F^{\pm}, G^{\pm}\in C^r$ and having $(k,l)$-$\Sigma$-center at the origin, where $\Sigma$ is the $x$-axis. Then there exists a $C^r$ diffeomorphism $h$ from period annulus of (\ref{p151}) to a period annulus of (\ref{p1500}),  which maps $x$-axis to the $x$-axis.
   \end{pr}
   
   We note that the system (\ref{p1500}) is  piecewise smooth Hamiltonian system with Hamiltonian $H^+(x,y)=-y-x^2, y> 0$ and $H^-(x,y)=y-x^2,y< 0$.
   The following proposition gives the information about the limit cycles bifurcated from period annulus of this system due to perturbation of switching manifold.
   \begin{pr}\label{p15P2}
   	The number of limit cycles for the system $\dot{X}=\begin{cases}
   	(-1,2x)& \mbox{if}~y>\varepsilon f(x)\\
   	(1,2x)& \mbox{if}~ y<\varepsilon f(x)
   	\end{cases}$ is equal to the number of isolated positive zeros of $f_o(x)$, where $f_o(x)=\dfrac{f(x)-f(-x)}{2}$.
   	
   \end{pr}
   \begin{proof} From (\ref{p143}) and (\ref{p144}), we get 
   \begin{align}
   M_1(r)=& 2[f(-{r})-f({r})]=4f_o({r}),~~ \mbox{and}\nonumber \\
   M_2(r)=& \dfrac{f'({r})}{{r}}M_1(r).\nonumber
   \end{align}
   Therefore the number of limit cycles bifurcated from period annulus of the unperturbed system is same as the number of positive roots of $y=f_o(x)$. In particular, if $f$ is an even function, then no limit cycles bifurcated.\end{proof}
   \begin{remark}
   	From Proposition \ref{p13P1}, it is clear that the limit cycle of system (\ref{p1500}) through the point $({r},0)$  is stable if and only if $\dfrac{dM_1}{dr}=\dfrac{-2}{{r}}{f_e}'({r})<0$, where $f_e(x)=\dfrac{f(x)+f(-x)}{2}.$
   \end{remark}
   \par Since the system (\ref{p1500}) is the normal form of planar piecewise smooth systems of the center-center type, Proposition (\ref{p15P2}) also holds for the systems (\ref{p151}).
\subsection {Boundary perturbation of saddle-center type system} 
In \cite{zou2018piecewise,zou2019piecewise}, authors studied the number of limit cyclce alongwith stability and hyperbolicity  of limit cycles of  the system
   \begin{align}\label{p153} \dot{X}=\begin{cases}
 (y-a, x)& \mbox{if}~y>\varepsilon f(x) \\
 (-y,x)& \mbox{if}~y<\varepsilon f(x)
 \end{cases}.\end{align} Origin is singularity of this system of saddle-center type  and $f(x)=\sin x$ or $f(x)=x(x^2-x_1^2)(x^2-x_2^2)...(x^2-x_m^2)$. Also, note that this system is piecewise linear Hamiltonian system with two zones.
 We study the number of limit cycles bifurcated from the period annulus  and stability of above systems if we change the switching manifold $y=0$ to $y=\varepsilon f(x)$, under the assumption that $f$ is sufficiently smooth.  
Using dialtion $x=au, y=av$ and renaming the variables $u$ and $v$ by $x$ and $y$ respectively, system (\ref{p151}) becomes
  \begin{align}\label{p154}
  \dot{X}=&\begin{cases}
  (y-1,x),~~~~ y>\varepsilon f(ax)\\
  (-y,x), ~~~~~~~y<\varepsilon f(ax)
  \end{cases}.
  \end{align}
 At $\varepsilon=0$, system (\ref{p153}) has a period annulus  $\displaystyle \mathcal{A}=\bigcup_{r\in [0,1]}(\Gamma_r^+\cup \Gamma_r^-)$, where \begin{align}\Gamma_r^+:H^+(x,y)=\dfrac{(y-1)^2}{2}-\dfrac{x^2}{2}=\frac{r}{2}, y\geq 0 ~~ \mbox {and}~~ \Gamma_r^-:H^-(x,y)=\dfrac{x^2}{2}+\dfrac{y^2}{2}=\frac{s}{2}, y\leq 0. \nonumber
\end{align} 
Now for any point $P(r)=(p(r),0), r\in [0,1]$  we have
\begin{align}H^+(P) =\dfrac{r}{2}=H^-(P)=\dfrac{s}{2}\Rightarrow \dfrac{(0-1)^2}{2}-\dfrac{(p(r))^2}{2}=\dfrac{r}{2}=-\dfrac{(p(r))^2}{2}-\dfrac{0^2}{2}=\dfrac{s}{2}\Rightarrow p(r)=\sqrt{1-r}=\sqrt{s}. \nonumber\end{align} Since the trajectories are symmetric about the $y$-axis, we have
$p_1(r)=-p(r)=-\sqrt{1-r} $. Melnikov functions for (\ref{p153}) are given by the following proposition.

 \begin{figure}[h]
	\centering
	\begin{subfigure}{.4\textwidth}
		\includegraphics[width=.45\linewidth]{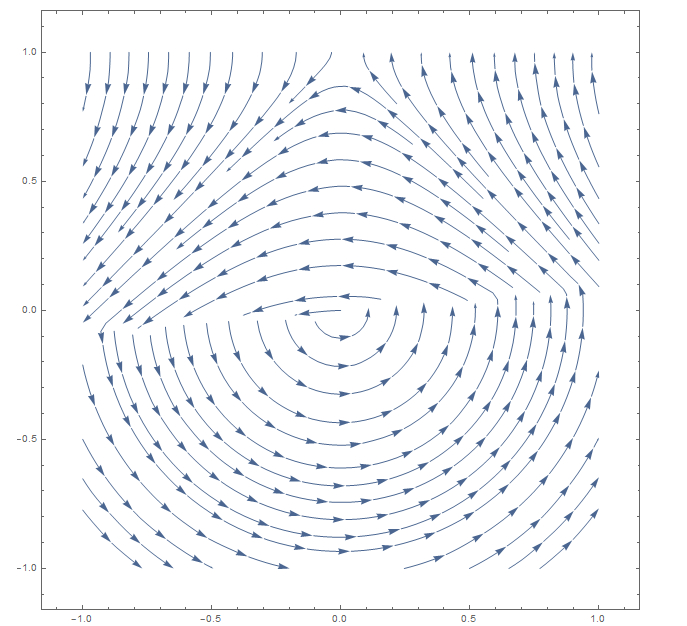}
		\caption{ $f(x)=0.01 \sin x$}
		\label{fig52a}
	\end{subfigure}%
	\begin{subfigure}{.4\textwidth}
		\includegraphics[width=.4\linewidth]{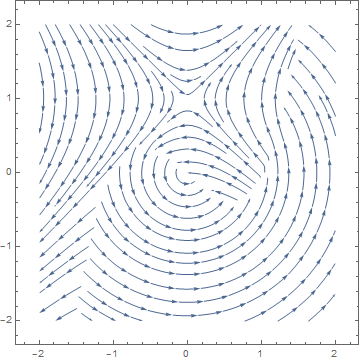}
		\caption{ $f(x)=x(x^2-1)(x^2-4)$}
		\label{fig52b}
	\end{subfigure}%
	
	\caption{Saddle-center system with perturbation boundary $y=\varepsilon f(x)$ }
	\label{fig3}
\end{figure}

\begin{pr}\label{p15P3}
	Suppose that $f:\R\rightarrow \R$ is $C^1$ smooth function. Then we have the following:
	\begin{enumerate} 
		\item The first order Melnikov function for the system (\ref{p153}) is given by $$M_1(r)=2f_o(a\sqrt{1-r}),$$ where  $f_o(x)=\dfrac{f(x)-f(-x)}{2}$ for $0<r<1$. Further, if $M_1(r)\equiv 0$, then $M_2(r)\equiv 0$.
		
		\item The number of limit cycles bifurcated from period annulus 
	around the origin inside the homoclinic orbit containing the saddle point $(0,a)$, is the number of isolated positive roots of $f_o(x)$. In particular, if $f$ is an even function, then no limit cycles bifurcated from the period annulus.
	\end{enumerate}
\end{pr}
\begin{proof} 
From (\ref{p143}), we have
 \begin{align}
 M_1(r)=\dfrac{(-p(r))(-p_1(r))}{(-p(r))}\left(\dfrac{-1}{(-p_1(r))}-\dfrac{0}{(-p_1(r))}\right)(f(ap(r))-f(ap_1(r)))=f(a p(r))-f(-a p(r)).\nonumber
 \end{align}
Therefore 
$M_1(r)=2f_0(ap(r))=2f_0(a\sqrt{1-r})$ for all $r\in [0,1]$.\\
Here, $H_{xx}^{\pm}=H_{yy}^{\pm}=H_{xy}^{\pm}=0$. Therefore from (\ref{p144}), we get $M_2(r)=\dfrac{-f'(a\sqrt{1-r})}{\sqrt{1-r}}M_1(r),r\in [0,1).$ Hence the proof of (1).
\par Proof of (2) follows from the expression for $M_1$. \end{proof}
\begin{remark}
	From Proposition \ref{p13P1}, it is clear that the limit cycle of (\ref{p154}) through the point $(\sqrt{1-r},0)$  is stable if and only if $\dfrac{dM_1}{dr}=\dfrac{-2a}{\sqrt{1-r}}{f_e}'(\sqrt{1-r})<0$, where $f_e(x)=\dfrac{f(x)+f(-x)}{2}.$
\end{remark}

Similarly we can characterize all planar piecewise smooth differential systems having saddle-center type as stated in the following proposition.

\begin{pr}\label{p15P4}
	If the Filippov system \begin{align}\label{p155}\dot{X}=\begin{cases}
	(F^+,G^+), ~y>0\\ (F^-,G^-), ~y<0
	\end{cases}\end{align} has a period annulus around the origin inside the homoclinic orbit containing the saddle point $(0,a), a>0$, then there is a homeomorphism which maps the period annulus of (\ref{p153}) to the period annulus of (\ref{p155}) and maps switching manifold to switching manifold. 
\end{pr}
\begin{proof} 
Let $ U^+ $ be an open region in the upper half plane lying inside the  homoclinic connection  containing the saddle point $(0,a)$  of
\begin{align}\label{p156} \dot{X}=(F^+, G^+).\end{align} Let  $(b,0)$ and $(c,0)$ be the points of intersection of the homoclinic orbit of (\ref{p156}) with the $x$-axis and $b<0<c$. We may assume that the periodic orbits of (\ref{p156}) are convex (see \cite{wei2016normal}). Hence we use the polar co-ordinates $(r,\theta)$
to transform  the system  (\ref{p156}) into \begin{align}\label{p158} \frac{dr}{d\theta}=P^+(r,\theta).\end{align}
Let $0<r_0<c$ and $0<\theta <\pi$. Consider the initial value problem $\dfrac{dr}{d\theta}=P^+(r,\theta), r(0)=r_0$. Let $\xi= \xi(\theta, r_0)$ be its solution. 
Now define the function $\Phi^+:[0,\pi]\times [0,c]\rightarrow U^+$ by \begin{align}\label{p159} \Phi^+(\theta,r)=(\xi(\theta,r)\cos\theta, \xi(\theta,r)\sin\theta).\end{align} Then $\Phi ^+$ is a diffeomorphism and maps each horizontal line segment $r=r_0$ to the trajectory $\xi=\xi(\theta , r_0)$ of (\ref{p156}). 

\par Similarly, if $V^+$ denote the region in the upper half plane occupied by the periodic orbits of 
\begin{align}\label{p1510} \dot{X}=(y-1, x),\end{align} then we have the
diffeomorphism $\Psi^+(\theta, r):[0,\pi]\times [0,1]\rightarrow V^+$. Let $\chi^+ :[0,\pi]\times [0,c]\rightarrow [0,\pi]\times [0,1]$ be the map given by $\chi^+(x,y)=(x,y/c)$. Then $\chi^+$ is also a diffeomorphism.
\par The composition $\mathcal{H}^+:=\Psi^+\circ \chi^+ \circ (\Phi^+)^{-1}:U^+\rightarrow V^+$ is a diffeomorphism.
\par Next, let $U^-$ and $V^-$ denote the open regions in lower half plane consisting of orbits of the sysetms \begin{align}\label{p1511} \dot{X}=(F^-,G^-)\end{align}
and
\begin{align}\label{p1512} \dot{X}=(-y,x),\end{align} respectively.
Then we can construct a diffeomorphism 
$\mathcal{H}^-:U^-\rightarrow V^-$ which maps orbits of (\ref{p1511}) to that of (\ref{p1512}).
\par Since the system (\ref{p155}) is Filippov, every point on the switching manifold $y=0$ is a singularity of order one.  Hence, $\displaystyle \lim_{y\rightarrow 0^+}\mathcal{H}^+(x,y)=\lim_{y\rightarrow 0^-}\mathcal{H}^-(x,y)$. \par Now we define the map $$\mathcal{H}:U^+\cup U^-\cup \{(x,0):b<x<0~\mbox{or}~0<x<c\}\rightarrow V^+\cup V^-\cup \{(x,0):0<|x|<1\}$$ by
\begin{align} \mathcal{H}(x,y)=\begin{cases}\mathcal{H}^+(x,y),(x,y)\in U^+ \\
\mathcal{H}^-(x,y),(x,y)\in U^-\\
\displaystyle{\lim_{y\rightarrow 0^+}\mathcal{H}^+(x,y)}=\lim_{y\rightarrow 0^-}\mathcal{H}^-(x,y), y=0 \\
\end{cases}.\end{align}
Note that due to the Filippov convension, $\displaystyle \lim_{y\rightarrow 0^+}F^+(x,y)=\lim_{y\rightarrow 0^-}F^-(x,y)$ and $\displaystyle \lim_{y\rightarrow 0^+}G^+(x,y)=\lim_{y\rightarrow 0^-}G^-(x,y)$, so that $\mathcal{H}$ is continuously differentiable on the switching manifold $y=0$. Therefore $\mathcal{H}$ is a diffeomorphism.
\end{proof}
\par From Proposition (\ref{p15P4}) we conclude that the Proposition (\ref{p15P3}) holds for the system (\ref{p155}).

\section{Concluding Remark}
In this article we found expressions for first order as well as second order Melnikov functions for perturbed planar piecewise smooth Hamiltonian systems. Using Melnikov functions we study limit cycle bifurcations of piecewise smooth Hamiltonian systems due to the perturbation of the switching manifold.
\par This idea could be extended to study limit bifurcation of any piecewise  smooth planar differential system. 
\section*{ }
	\bibliographystyle{elsarticle-num-names-nourl}
	\bibliography{Ref1}
\end{document}